\documentclass[a4paper,12pt]{article}

\usepackage[left=2cm,right=2cm, top=2cm,bottom=2cm,bindingoffset=0cm]{geometry}

\usepackage{verbatim}
\usepackage{amsmath}
\usepackage{amsthm}
\usepackage{amssymb}
\usepackage{delarray}
\usepackage{cite}
\usepackage{hyperref}
\usepackage{mathrsfs}

\newcommand{\ga}{\gamma}

\newcommand{\la}{\lambda}

\newcommand{\iy}{\infty}

\theoremstyle{plain}

\newtheorem{thm}{Theorem}
\newtheorem{lem}{Lemma}

\theoremstyle{definition}

\theoremstyle{remark}

 \allowdisplaybreaks

\begin{document}

\begin{center}
{\large\bf An inverse problem for the integro-differential Dirac system with partial information given on the convolution kernel
}
\\[0.2cm]
{\bf Natalia P. Bondarenko} \\[0.2cm]
\end{center}

\vspace{0.5cm}

{\bf Abstract.} An integro-differential Dirac system with an integral term in the form of convolution is considered.
We suppose that the convolution kernel is known a priori on a part of the interval, and recover it on the remaining part,
using a part of the spectrum. We prove the uniqueness theorem, provide an algorithm for the solution of the inverse problem
together with necessary and sufficient conditions for its solvability. 

\medskip


\medskip


\vspace{1cm}

{\large \bf 1. Introduction}

\bigskip

In this paper, we study an inverse spectral problem for the integro-differential Dirac system
\begin{equation} \label{eqv}
  B y' + \int_0^x M(x - t) y(t) \, dt = \la y, \quad 0<x<\pi,
\end{equation}
where $\la$ is the spectral parameter,
$$
    B = \begin{pmatrix} 0 & 1 \\ -1 & 0 \end{pmatrix}, \quad
    M(x) = \begin{pmatrix} p(x) & q(x) \\ -q(x) & p(x) \end{pmatrix}, \quad
    y(x) = \begin{pmatrix} y_1(x) \\ y_2(x) \end{pmatrix},
$$
the functions $p$, $q$ are complex-valued and belong to the class $L_{2, \pi}(0, \pi)$, 
$$
L_{2, \pi}(c, \pi) := \{ r \colon (\pi - x) r(x) \in L_2(c, \pi) \}, \quad c \ge 0.
$$

Inverse spectral problems consist in recovering differential operators from their spectral characteristics.
Classical results on inverse spectral problems for Sturm-Liouville and Dirac differential operators can be found in
the monographs \cite{Mar77, Lev84, LS91, PT87, FY01}. In recent years, inverse problems for integro-differential operators
attract much attention of mathematicians. Such operators are nonlocal and therefore more difficult for investigation, than differential operators. 
First results on inverse problems for integro-differential operators were obtained in the works \cite{Mal79, Er88, Yur91}.
In the paper \cite{But06}, a new method was proposed, based on the reduction of an inverse problem for a nonlocal operator to
a nonlinear integral equation, containing series of convolutions. Later on, this method has been successfully applied
for solution of inverse problems for Sturm-Liouville operators with an integral delay (see \cite{But07, But10, BR15}) and for the 
Dirac system in the form \eqref{eqv} (see \cite{BB17}).
 
We describe the results of the paper \cite{BB17} in more detail. 
Let $D = D(p, q)$ be the boundary value problem for the system \eqref{eqv} with the boundary conditions
\begin{equation} \label{bc}
   y_1(0) = y_1(\pi) = 0.
\end{equation}
The following inverse spectral problem has been investigated in \cite{BB17}.

\medskip

{\bf Inverse Problem 1.} {\it Given the spectrum of $D$, find $p$ and $q$.}

\medskip

The main results of \cite{BB17} can be formulated as follows.

\begin{thm} \label{thm:BB17}
\begin{enumerate}
The problem $D$ has a countable set of eigenvalues, which can be numbered as $\{ \la_k \}_{k \in \mathbb Z}$ counting 
with their multiplicities, so that
\begin{equation} \label{asymptla}
\la_k = k + \varkappa_k, \quad \{ \varkappa_k \} \in l_2.
\end{equation}

For any sequence of complex numbers $\{ \la_k \}_{k \in \mathbb Z}$ of the form \eqref{asymptla}, there exist unique functions $p, q \in L_{2, \pi}(0, \pi)$
such that $\{ \la_k \}_{k \in \mathbb Z}$ is the spectrum of the boundary value problem $D(p, q)$ of the form \eqref{eqv}-\eqref{bc}.
\end{enumerate}
\end{thm}

The proof of Theorem~\ref{thm:BB17} is constructive, and gives an algorithm for the solution of Inverse Problem~1.

In this paper, we suppose that the kernel of the integral term $M(x)$ is known a priori on the subinterval $(0, a)$, $\frac{\pi}{2} \le a < \pi$.
We show that in this case, only a fractional part of the spectrum is sufficient to reconstruct the matrix $M(x)$ on the subinterval $(a, \pi)$.
Such problem statement is similar in some sense to the Hochstadt-Lieberman problem \cite{HL78}, which consists in recovering
the potential of the Sturm-Liouville operator on a finite interval from a spectrum, while the potential on the half of the interval is known a priori.
The Hochstadt-Lieberman problem, also called the half-inverse problem, and its generalizations for various classes of differential operators and pencils
were studied in \cite{GS00, Sakh01, HM04, SY08, MP10, SBI11, But11, BS12, Yang10, Yang11, YZ12, WX09, WX12, Bond17-1, Bond17-2}.
However, we know the only paper \cite{BS16} on a half-inverse problem for an integro-differential operator, where the Sturm-Liouville operator
with an integral delay in the form of convolution is recovered from a half of the spectrum and 
the convolution kernel is given a priori on the half of the interval. We also mention the paper \cite{Bond18}, where an inverse problem
for the integro-differential Sturm-Liouville operator on a graph was studied under the assumption, that the operator coefficients are known
on a part of the graph.

The paper is organized as follows. In Section~2, we provide the system of the main equations for Inverse Problem~1, derived in \cite{BB17},
and show the benefits of the knowledge of the functions $p$ and $q$ on a subinterval. In Section~3, the partial inverse problem is formulated,
which consists in recovering the functions $p$ and $q$ on the remaining part of the interval from an abstract subset of the spectrum $\{ \la_k \}_{k \in \mathcal I}$,
$\mathcal I \subset \mathbb Z$.
We prove the uniqueness theorem in terms of completeness of some system of vector functions, associated with the given subspectrum.
We also provide a constructive algorithm for the solution of the partial inverse problem. In Section~4, we consider the case $a = \pi - \frac{\pi}{m}$,
$m \in \mathbb N$, $m \ge 2$,
$\mathcal I = \{ s m \colon s \in \mathbb Z \}$. In this case the partial inverse problem is uniquely solvable. Using the results of Section~3, we obtain
the necessary and sufficient conditions for the solvability of the inverse problem in this special case.  

\bigskip

{\large \bf 2. Main equations}

\bigskip

In this section, we modify the system of the main equations from the paper \cite{BB17}, assuming that the functions $p$ and $q$ are known on 
the subinterval $(0, a)$.

Let us introduce the following notations for the convolutions:
$$
   (f * g)(x) = \int_0^x f(t) g(x - t) \, dt, \quad f^{*n}(x) = \underbrace{f * f * \dots * f}_{n}(x), \quad f^{*0}(x) = \delta(x),
$$
where $\delta$ is the Dirac delta function: $\delta * f = f * \delta = f$. 

In the paper \cite{BB17}, Inverse Problem~1 was reduced to the system of nonlinear integral equations with respect 
to the functions $p$ and $q$:

\begin{equation} \label{intK}
\arraycolsep=1.4pt\def\arraystretch{2.2}
\left\{\begin{array}{l} \displaystyle -w_1(\pi-t) = (\pi - t) q(t) - \sum_{n = 2}^{\iy} \sum_{j = 0}^n
\dfrac{a_{nj} (\pi - t)^n }{n!} \left( p^{*j} * q^{*(n - j)} \right) (t),
\\
\displaystyle  -w_2(\pi-t) = (\pi - t) p(t) + \sum_{n = 2}^{\iy} \sum_{j = 0}^n \dfrac{b_{nj} (\pi -
t)^n }{n!} \left( p^{*j} * q^{*(n - j)} \right) (t).
\end{array}\right.
\end{equation}
Here $a_{nj}$ and $b_{nj}$ are certain coefficients, satisfying the estimates
$$
	\sum_{j = 0}^n |a_{nj}| \le 2^{n-1}, \quad \sum_{j = 0}^n |b_{nj}| \le 2^{n-1}, \quad n \in \mathbb N,
$$
and the functions $w_1$, $w_2$ can be found from the following representation of the characteristic function $\Delta(\la)$,
whose zeros coincide with the eigenvalues of $D$:
\begin{equation} \label{Delta}
\Delta(\la) = \sin \la \pi + \int_0^{\pi}\Big( w_{1}(t) \sin \la t + w_{2}(t) \cos \la t\Big) \, dt, \quad w_1, w_2 \in L_2(0,\pi).
\end{equation}

Suppose that the matrix-function $M(x)$ is known on the interval $(0, a)$. Introduce the functions
\begin{gather} \label{defp1}
   p_1(x) = \left\{ \begin{array}{ll} p(x), & x \in (0, a), \\ 0, & x \in (a, \pi), \end{array} \right. \quad
   q_1(x) = \left\{ \begin{array}{ll} q(x), & x \in (0, a), \\ 0, & x \in (a, \pi), \end{array} \right. \\ \nonumber
   p_2(x) = \left\{ \begin{array}{ll} 0, & x \in (0, a), \\ p(x), & x \in (a, \pi), \end{array} \right. \quad
   q_2(x) = \left\{ \begin{array}{ll} 0, & x \in (0, a), \\ q(x), & x \in (a, \pi). \end{array} \right. 
\end{gather}
Thus, $p(x) = p_1(x) + p_2(x)$, $q(x) = q_1(x) + q_2(x)$. One can easily show that
$$
   (p_1 + p_2)^{*j} * (q_1 + q_2)^{* (n - j)} = p_1^{*j} * q_1^{*(n - j)} + j p_1^{*(j - 1)} * q_1^{*(n - j)} * p_2 +
   (n - j) p_1^{*j} * q_1^{*(n - j - 1)} * q_2,
$$
so the system \eqref{intK} takes the form
\begin{equation} \label{intK2}
\arraycolsep=1.4pt\def\arraystretch{2.2}
\left\{\begin{array}{l} \displaystyle -w_1(\pi-t) = (\pi - t) q_1(t) + (\pi - t) q_2(t) +
A_1(t) + \sum_{n = 2}^{\iy} \frac{(\pi - t)^n}{n!} (A_{n1} * p_2 + A_{n2} * q_2)(t),
\\
\displaystyle  -w_2(\pi-t) = (\pi - t) p_1(t) + (\pi - t) p_2(t) + B_1(t) + 
\sum_{n = 2}^{\iy} \frac{(\pi - t)^n}{n!} (B_{n1} * p_2 + B_{n2}* q_2)(t), \end{array}\right.
\end{equation}
where
\begin{align} \label{defA1}
   A_1(t) & = -\sum_{n = 2}^{\iy} \sum_{j = 0}^n \frac{a_{nj} (\pi - t)^n}{n!} \left( p_1^{*j} * q_1^{*(n - j)} \right) (t), \\ \label{defA2}
   A_{n1}(t) & = -\sum_{j = 1}^n j a_{nj} \left(p_1^{*(j-1)}*q_1^{*(n-j)} \right)(t), \quad 
   A_{n2}(t) = -\sum_{j = 0}^{n-1} (n-j) a_{nj} \left(p_1^{*j}*q_1^{*(n-j-1)}\right)(t),
   \\ \label{defB1}
   B_1(t) & = \sum_{n = 2}^{\iy} \sum_{j = 0}^n \dfrac{b_{nj} (\pi - t)^n }{n!} \left( p_1^{*j} * q_1^{*(n - j)} \right) (t), \\ \label{defB2}
   B_{n1}(t) & = \sum_{j = 1}^n j b_{nj} \left( p_1^{*(j-1)}*q_1^{*(n-j)} \right)(t), \quad 
   B_{n2}(t) = \sum_{j = 0}^{n-1} (n - j) b_{nj} \left( p_1^{*j}*q_1^{*(n-j-1)} \right)(t).
\end{align}
For arbitrary functions $f, g \in L_2(0, \pi)$ their convolution $f * g$ also belong to $L_2(0, \pi)$, and
$$
    \| f * g \| \le \pi \| f \| \| g \|, \quad \| f^{*n} \| \le \pi^{n-1} \| f \|^n, \quad n \in \mathbb N,
$$
where $\| . \|$ is the norm in $L_2(0, \pi)$.
Consequently, the series $A_1$ and $B_1$ converge in $L_2(0,\pi)$, and the functions $A_{nk}$, $B_{nk}$ satisfy the estimate
\begin{equation} \label{estAB}
   \| A_{nk} \|, \| B_{nk} \| \le n \pi^{n-2} 2^{n-1} \max \{ \| p_1 \|, \| q_1 \| \}^{n-1}, \quad n \ge 2, \: k = 1, 2.
\end{equation}
For $t \in (0, a)$, the system \eqref{intK2} takes the form
\begin{equation} \label{intK3}
\left\{\begin{array}{l} -w_1(\pi-t) = (\pi - t) q_1(t) + A_1(t),\\
						 -w_2(\pi-t) = (\pi - t) p_1(t) + B_1(t), \end{array}\right.
\end{equation}
Thus, known the functions $p, q \in L_2(0, a)$, one can find $w_1, w_2 \in L_2(b, \pi)$, $b := \pi - a$. 

For $t \in (a, \pi)$, we have the system \eqref{intK2} with respect to the functions $p_2$ and $q_2$.

\begin{lem} \label{lem:uniq}
For arbitrary functions $w_1, w_2 \in L_2(0, b)$, $A_1, B_1 \in L_2(a, \pi)$ and $A_{nk}, B_{nk} \in L_2(0, \pi)$, $n \ge 2$, $k = 1, 2$, satisfying the estimate~\eqref{estAB},
the system \eqref{intK2} has a unique solution $p_2,q_2 \in L_{2,\pi}(a, \pi)$, i.e. $(\pi - t) p_2(t), \, (\pi - t) q_2(t) \in L_2(a, \pi)$.
\end{lem}

\begin{proof}
Denote
\begin{gather*}
   y(t) = (\pi - t) \begin{pmatrix} q_2(t) \\ p_2(t) \end{pmatrix}, \quad 
   \mu(t) = \begin{pmatrix} -w_1(\pi - t) - A_1(t) \\ -w_2(\pi - t) - B_1(t) \end{pmatrix}, \\
   H(t, s) = \frac{\pi-t}{\pi-s} \sum_{n = 2}^{\iy} \frac{(\pi-t)^{n-1}}{n!} 
   \begin{pmatrix} A_{n1}(t-s) & A_{n2}(t-s) \\ B_{n1}(t-s) & B_{n2}(t-s) \end{pmatrix}.
\end{gather*}
Then the system \eqref{intK2} for $t \in (a, \pi)$ can be rewritten in the form
\begin{equation} \label{Volt}
   y(t) + \int_a^t H(t, s) y(s) \, ds = \mu(t).
\end{equation}
According to the conditions of the theorem, $\mu \in L_2(a, \pi) \oplus L_2(a,\pi)$ and the entries of the matrix function
$H(t, s)$ belong to $L_2\left( (a, \pi)\times (a, \pi) \right)$. Thus, the Volterra integral equation \eqref{Volt} of the second kind has a unique solution
$y \in L_2(a, \pi) \oplus L_2(a,\pi)$.
\end{proof}

\bigskip

{\large \bf 3. Abstract partial inverse problem}

\bigskip

Let $\mathcal I$ be a subset of $\mathbb Z$. In this section, we study the following partial inverse problem.

\medskip

{\bf Inverse Problem 2.} {\it Given the functions $p, q$ on the interval $(0, a)$ and the subset of the spectrum $\{ \la_k \}_{k \in \mathcal I}$,
construct $p$ and $q$ on the interval $(a, \pi)$.}

\medskip

Of course, the uniqueness of the solution and the possibility of the constructive solution of Inverse Problem~2 depend on $a$ and a subset $\mathcal I$.
Note that the subset $\{ \la_k \}_{k \in \mathcal I} $ can contain multiple eigenvalues. Denote by $\{ \la_k \}_{k \in \mathcal J}$ the subset of all the
distinct eigenvalues among $\{ \la_k \}_{k \in \mathcal I}$, and let $\mathcal M := \{ m_k \}_{k \in \mathcal J}$ be the corresponding multiplicities,
$m_k \ge 1$. 

The analysis in the previous section shows, that we need to know $w_1$ and $w_2$ on the interval $(0, b)$, in order to find $p_2$ and $q_2$.
Introduce the function
\begin{equation} \label{defE}
   E(\la) := \Delta(\la) - \sin \la \pi - \int_b^{\pi} \left( w_1(t) \sin \la t + w_2(t) \cos \la t\right) \, dt.
\end{equation}
Relation \eqref{Delta} implies
\begin{equation} \label{relE}
   E(\la) = \int_0^b \left( w_1(t) \sin \la t + w_2(t) \cos \la t \right) \, dt.
\end{equation}
It follows from \eqref{defE} and \eqref{relE}, that
\begin{align} \label{defEk}
   \frac{d^j}{d \la^j} E(\la_k) & = -\left( \frac{d^j}{d \la^j} \sin \la \pi\right)_{\la = \la_k} -\left(\frac{d^j}{d \la^j} \int_b^{\pi} (w_1(t) \sin \la t + w_2(t) \cos \la t) \, dt \right)_{\la = \la_k}, 
    \\ \label{relEk}
    \frac{d^j}{d \la^j} E(\la_k) & = \int_0^b \left(w_1(t) \bigg( \frac{d^j}{d \la^j} \sin \la t\bigg)_{\la = \la_k} + w_2(t) \bigg(\frac{d^j}{d \la^j} \cos \la t\bigg)_{\la = \la_k} \right) \, dt
\end{align}
for $k \in \mathcal J, \: j = \overline{0, m_k-1}$,  
since $\dfrac{d^j}{d \la^j} \Delta(\la_k) = 0$ for such $k$.

Consider the complex Hilbert space $\mathcal H := L_2(0, b) \oplus L_2(0, b)$ with the scalar product 
$$
   (g, h) = \int_0^b (\overline{g_1(t)} h_1(t) + \overline{g_2(t)} h_2(t) )\, dt, \quad g, h \in \mathcal H, \quad
   g = \begin{pmatrix} g_1 \\ g_2 \end{pmatrix},
   h = \begin{pmatrix} h_1 \\ h_2 \end{pmatrix}.
$$
In view of \eqref{relEk}, the functions $w_1$ and $w_2$ on $(0, b)$ are uniquely specified by the values of $\dfrac{d^j}{d \la^j}E(\la_k)$, $k \in \mathcal J$,
$j = \overline{0, m_k-1}$, if the system of vector functions 
\begin{equation} \label{SI}
    S_{\mathcal I} := \left\{ \begin{pmatrix} \dfrac{d^j}{d \la^j} \sin \la t \\[10pt] \dfrac{d^j}{d \la^j} \cos \la t\end{pmatrix}_{\la = \la_k} \right\}_{k \in \mathcal J, \, j = \overline{0, m_k-1}}
\end{equation}
is complete in $\mathcal H$. The values $\dfrac{d^j}{d \la^j}E(\la_k)$ can be determined by \eqref{defEk}.

Along with $D$, we consider another boundary value problem $\tilde D = D(\tilde p, \tilde q)$ of the same form, but with
different coefficients. We agree that if a certain symbol $\ga$ denotes an object related to $D$, the corresponding symbol
$\tilde \ga$ denotes an analogous object related to $\tilde D$. Now we are ready to formulate a uniqueness theorem for Inverse Problem~2.

\begin{thm} \label{thm:uniq}
Let $p(x) = \tilde p(x)$ and $q(x) = \tilde q(x)$ a.e. on $(0, a)$, and let the eigenvalues
$\{ \la_k \}_{k \in \mathbb Z}$ and $\{ \tilde \la_k \}_{k \in \mathbb Z}$ be numbered in such a way, that the asymptotic formula
\eqref{asymptla} holds for the both sequences, and $\la_k = \tilde \la_k$ for $k \in \mathcal I$, where the subset $\mathcal I$ is such that 
the system $\mathcal S_{\mathcal I}$ is complete in $\mathcal H$. Then $p(x) = \tilde p(x)$ and $q(x) = \tilde q(x)$ a.e. on $(a, \pi)$.
\end{thm} 

\begin{proof}
Since $p(x) = \tilde p(x)$ and $q(x) = \tilde q(x)$ a.e. on $(0, a)$, we have $A_1(t) = \tilde A_1(t)$,
$B_1(t) = \tilde B_1(t)$, $A_{nk}(t) = \tilde A_{nk}(t)$, $B_{nk}(t) = \tilde B_{nk}(t)$ for $n \ge 2$, $k = 1, 2$,
$t \in [0, \pi)$. The relations \eqref{intK3} yield
$w_1(x) = \tilde w_1(x)$, $w_2(x) = \tilde w_2(x)$ a.e. on $(b, \pi)$.
By virtue of \eqref{defEk} and the equality 
$\la_k = \tilde \la_k$, $k \in \mathcal I$, we have $\dfrac{d^j}{d \la^j} E(\la_k) = \dfrac{d^j}{d \la^j}\tilde E(\tilde \la_k)$
for $k \in \mathcal J$, $j = \overline{0, m_k-1}$.
Since the system $\mathcal S_{\mathcal I}$ is complete in $\mathcal H$, we obtain $w_1(t) = \tilde w_1(t)$ 
and $w_2(t) = \tilde w_2(t)$ a.e. on $(0, b)$. Applying Lemma~\ref{lem:uniq}, we arrive at the assertion of the theorem.
\end{proof}

If the system $\mathcal S_{\mathcal I}$ is a Riesz basis in $\mathcal H$, one can solve Inverse Problem~2 by the following algorithm.

\medskip

{\bf Algorithm 1.} Let the functions $p(x)$ and $q(x)$ for $x \in (0, a)$ and the eigenvalues $\{ \la_k \}_{k \in \mathcal I}$ be given.

\begin{enumerate}
\item Construct the functions $p_1$, $q_1$ by \eqref{defp1} and then $A_1$, $B_1$ by \eqref{defA1}, \eqref{defB1}.
\item Find $w_1(t)$ and $w_2(t)$ for $t \in (b, \pi)$ by formulas \eqref{intK3}.
\item Find $\dfrac{d^j}{d \la^j} E(\la_k)$, $k \in \mathcal J$, $j = \overline{0, m_k-1}$, using \eqref{defEk}.
\item Construct the system of vector functions $S_{\mathcal I}$, and determine the functions $w_1(t)$ and $w_2(t)$ from the coordinates $\dfrac{d^j}{d \la^j} E(\la_k)$ with respect to the Riesz basis
$S_{\mathcal I}$ (see \eqref{relEk}).
\item Construct the functions $A_{nk}$, $B_{nk}$, $n \ge 2$, $k = 1, 2$ by \eqref{defA2}, \eqref{defB2}.
\item Solving the system of Volterra integral equations \eqref{intK2}, find the functions $p_2(x) = p(x)$
and $q_2(x) = q(x)$ for $x \in (a, \pi)$.
\end{enumerate}

{\large \bf 4. Specific partial inverse problem}

\bigskip

The formulation of the results in the previous section contain abstract conditions of the completeness or the Riesz-basicity of the system $S_{\mathcal I}$.
In this section, we show that these conditions hold for some certain subsets $\mathcal I$ with the certain $a$.

Let $m \in \mathbb N$, $m \ge 2$. Put $a := \pi - \frac{\pi}{m}$, $\mathcal I := \{ s m \colon s \in \mathbb Z \}$.
Then the following result is valid.

\begin{lem} \label{lem:Riesz}
For an arbitrary sequence $\{ \la_k \}_{k \in \mathcal I} = \{ \la_{s m} \}_{s \in \mathbb Z}$ of not necessarily distinct complex numbers, satisfying
the relation
\begin{equation} \label{asymptlam}
    \la_{s m} = s m + \varkappa_s, \quad \{ \varkappa_s \} \in l_2,
\end{equation}
the system $S_{\mathcal I}$, defined by \eqref{SI}, is a Riesz basis in $\mathcal H$.
\end{lem}

\begin{proof}
According to the asymptotic relation \eqref{asymptlam}, the system $\mathcal S_{\mathcal I}$ is $l_2$-close to the system
$$
  	\left\{ \begin{pmatrix} \sin s m t \\ \cos s m t \end{pmatrix} \right\}_{s \in \mathbb Z},
$$
which is an orthogonal basis in $\mathcal H$. Thus, it remains to prove, that $\mathcal S_{\mathcal I}$ is complete in $\mathcal H$.
Suppose $h_1, h_2 \in L_2(0, b)$ are such functions, that
$$
   \int_0^b \left( h_1(t) \bigg(\dfrac{d^j}{d \la^j} \sin \la t\bigg)_{\la = \la_k} + 
   h_2(t) \bigg(\dfrac{d^j}{d \la^j} \cos \la t \bigg)_{\la = \la_k} \right) \, dt = 0, \quad k \in \mathcal J, \: j = \overline{0, m_k-1}.
$$
This means that the entire function
$$
    H(\la) := \int_0^b \left( h_1(t) \sin \la t + h_2(t) \cos \la t \right)\, dt
$$
has zeros $\la_k$ with the multiplicities $m_k$ for $k \in \mathcal J$.

On the other hand, in view of \eqref{asymptlam} and Theorem~\ref{thm:BB17}, the sequence $\{ \la_{sm} \}_{s \in \mathbb Z}$
is the spectrum of some boundary value problem $D$ of the form \eqref{eqv}-\eqref{bc} on the interval $\left(0, \frac{\pi}{m} \right)$
instead of $(0, \pi)$. Its characteristic function admits the following representation (see \eqref{Delta}):
$$
   \Delta(\la) = \sin \la b + \int_0^b \left( w_1(t) \sin \la t + w_2(t) \cos \la t \right) \, dt, \quad w_1, w_2 \in L_2(0, b). 
$$
Clearly, the function $\dfrac{H(\la)}{\Delta(\la)}$ is entire and bounded in the whole $\la$-plane. By the Liouville's theorem, 
$H(\la) \equiv C \Delta(\la)$, where $C$ is a constant. However, $\lim\limits_{\la \to +\iy} H(\la) = 0$, while 
$\lim\limits_{\la \to +\iy} \Delta(\la)$ does not exist. Thus, $C = 0$ and $H(\la) \equiv 0$. Consequently, $h_1(t) = h_2(t) = 0$
a.e. on $(0, b)$, and the system $\mathcal S_{\mathcal I}$ is complete.
\end{proof}

The following theorem provides sufficient (and also necessary) conditions for the solvability of Inverse Problem~2 for $a = \pi - \frac{\pi}{m}$, 
$I = \{ s m \colon s \in \mathbb Z \}$.

\begin{thm}
Let $p$ and $q$ be arbitrary functions from $L_2(0, a)$, $a = \pi - \frac{\pi}{m}$, and  
let $\{ \la_k \}_{k \in \mathcal I} = \{ \la_{s m} \}_{s \in \mathbb Z}$ be arbitrary complex numbers, satisfying \eqref{asymptlam}.
Then the functions $p$ and $q$ can be uniquely continued to the interval $(a, \pi)$, so that $p, q \in L_{2, \pi}(0, \pi)$ and $\{ \la_k \}_{k \in \mathcal I}$
are eigenvalues of the boundary value problem $D = D(p, q)$ with the corresponding multiplicities. 
\end{thm}

\begin{proof}
Since the system $\mathcal S_{\mathcal I}$, constructed from the given $\{ \la_k \}_{k \in \mathcal I}$, is a Riesz basis by Lemma~\ref{lem:Riesz},
one can construct the functions $p, q \in L_{2, \pi}(a, \pi)$, using Algorithm~1. By construction, $\{ \la_k \}_{k \in \mathcal I}$ are eigenvalues of the problem $D = D(p, q)$.
The uniqueness of $D$ follows from Theorem~\ref{thm:uniq}.
\end{proof}

For $m = 2$, we get a half inverse problem, when the functions $p$ and $q$ are known a priori on the interval $\left( 0, \frac{\pi}{2} \right)$
and can be recovered on $\left( \frac{\pi}{2}, \pi \right)$ by the half of the spectrum $\{ \la_{2k} \}_{k \in \mathbb Z}$.

\medskip
 
{\bf Acknowledgment.} This work was supported by Grant 17-11-01193 of the Russian Science Foundation.

\medskip

\noindent Natalia Pavlovna Bondarenko \\
1. Department of Applied Mathematics, Samara National Research University, \\
Moskovskoye Shosse 34, Samara 443086, Russia, \\
2. Department of Mechanics and Mathematics, Saratov State University, \\
Astrakhanskaya 83, Saratov 410012, Russia, \\
e-mail: {\it BondarenkoNP@info.sgu.ru}


\medskip

\begin{thebibliography}{99}

\bibitem{Mar77}
Marchenko, V. A. Sturm-Liouville Operators and their Applications, Naukova Dumka,
Kiev (1977) (Russian); English transl., Birkhauser (1986).

\bibitem{Lev84}
Levitan, B. M. Inverse Sturm-Liouville Problems, Nauka, Moscow (1984) (Russian); English
transl., VNU Sci. Press, Utrecht (1987).

\bibitem{LS91} 
Levitan, B. M.; Sargsyan I. S. Sturm-Liouville and Dirac Operators, Nauka, Moscow (1988);
English transl., Kluwer Academic Publishers, Dordrecht (1991).

\bibitem{PT87}
P\"{o}schel, J.; Trubowitz, E. Inverse Spectral Theory, New York, Academic Press (1987).

\bibitem{FY01}
Freiling, G.; Yurko, V. Inverse Sturm-Liouville problems and their applications. Huntington,
NY: Nova Science Publishers, 305 p. (2001).

\bibitem{Mal79}
Malamud, M. M. On some inverse problems, Boundary Value Problems
of Mathematical Physics, Kiev (1979), 116-–124.

\bibitem{Er88}
Eremin, M. S. An inverse problem for a second-order integro-differential equation
with a singularity, Differ. Uravn. 24:2 (1988), 350--351.

\bibitem{Yur91}
Yurko, V. A. Inverse problem for integrodifferential operators, Mathematical Notes 50:5 (1991), 1188--1197.

\bibitem{But06}
Buterin, S. A. The inverse problem of recovering the Volterra convolution operator from the incomplete spectrum of its 
rank-one perturbation, Inverse Problems 22 (2006), 2223--2236.

\bibitem{But07}
Buterin, S. A. On an inverse spectral problem for a convolution integro-differential operator,
Res. Math. 50 (2007), no.3--4, 73--181.

\bibitem{But10}
Buterin, S. A. On the reconstruction of a convolution perturbation of the Sturm-Liouville
operator from the spectrum, Diff. Uravn. 46 (2010), 146--149 (Russian); English transl. in
Diff. Eqns. 46 (2010), 150--154.

\bibitem{BR15}
Buterin, S. A.; Choque Rivero, A. E. On inverse problem for a convolution integrodifferential
operator with Robin boundary conditions, Appl. Math. Lett. 48 (2015) 150--155.

\bibitem{BB17}
Bondarenko, N.; Buterin, S. On Recovering the Dirac Operator with an Integral Delay from the Spectrum, Results Math. 71 (2017), 1521--1529.

\bibitem{HL78}
Hochstadt, H.; Lieberman, B. An inverse Sturm-Liouville problem with mixed given data, SIAM J. Appl. Math. 34 (1978), 676--680.

\bibitem{GS00}
Gesztesy, F.; Simon, B. Inverse spectral analysis with partial information on the potential, II. The case of discrete spectrum,
Trans. AMS 352:6 (2000), 2765--2787. 

\bibitem{Sakh01}
Sakhnovich, L. Half-inverse problems on the finite interval, Inverse Problems 17 (2001), 527--532.

\bibitem{HM04}
Hryniv, R.O.; Mykytyuk Ya. V. Half-inverse spectral problems for Sturm–Liouville operators with singular potentials, Inverse Probl. 20 (2004), 1423--1444.

\bibitem{SY08}
Shieh, C.-T.; Yurko, V. A. Inverse nodal and inverse spectral problems for discontinuous boundary value problems,
J. Math. Anal. Appl. 347 (2008), 266-272.

\bibitem{MP10}
Martinyuk, O.; Pivovarchik, V. On the Hochstadt-Lieberman theorem, Inverse Problems 26 (2010), 035011 (6pp).

\bibitem{SBI11}
Shieh, C.-T.; Buterin, S. A.; Ignatiev, M. Yu. On Hochstadt-Lieberman theorem for Sturm-Liouville operators, Far East J. Appl. Math. 52:2 (2011), 131--146.

\bibitem{But11}
Buterin, S. A. On half inverse problem for differential pencils with the spectral parameter in boundary conditions, Tamkang J. Math. 42:3 (2011), 355--364.

\bibitem{BS12}
Buterin, S.A.; Shieh C.-T. Incomplete inverse spectral and nodal problems for differential
pencils, Results Math. 62:1 (2012), 167--179.

\bibitem{Yang10}
Yang, C.-F. Inverse spectral problems for the Sturm-Liouville operator on a $d$-star graph, 
J. Math. Anal. Appl. 365 (2010), 742--749.

\bibitem{Yang11}
Yang, C.-F. Hochstadt--Lieberman theorem for Dirac operator with eigenparameter dependent boundary conditions,
Nonlinear Analysis: Theory, Methods and Applications 74:7 (2011), 2475--2484.

\bibitem{YZ12}
Yang, C.-F.; Zettl, A. Half inverse problems for quadratic pencils of Sturm-Liouville operators, Taiwanese J. Math. 16:5 (2012), 1829--1846.

\bibitem{WX09}
Wei, G.; Xu, H.-K. On the missing eigenvalue problem for an inverse Sturm–Liouville problem, Journal de Mathematiques Pures
et Appliquees, 91:5 (2009), 468--475.

\bibitem{WX12}
Wei, G.; Xu, H.-K. Inverse spectral problem with partial information given on the potential and norming constants,
Trans. Amer. Math. Soc. 364 (2012), 3265--3288. 

\bibitem{Bond17-1}
Bondarenko, N.P. Partial inverse problems for the Sturm-Liouville operator on a star-shaped graph with mixed boundary conditions,
J. Inverse and Ill-Posed Probl. (2018), Vol. 26, Issue 1. Pp. 1--12.

\bibitem{Bond17-2}
Bondarenko, N.P. A partial inverse problem for the Sturm-Liouville operator on a star-shaped graph,
Analysis and Mathematical Physics, published online 24 April 2017, pp. 1-14. DOI: 10.1007/s13324-017-0172-x

\bibitem{BS16} 
Buterin, S.; Sat, M. On the half inverse spectral problem for an integro-differential operator, Inverse Problems in Science and Engineering (2017),
Vol 25, Issue 10. Pp. 1508--1518.

\bibitem{Bond18}
Bondarenko, N.P. An inverse problem for an integro-differential operator on a star-shaped graph, Mathematical Methods in the Applied Sciences (2018), 
Vol. 41, Issue 4. Pp. 1697-1702.

\end{thebibliography}
\end{document}